\documentclass[12pt]{amsart}

\usepackage{latexsym, amssymb, amsmath}

\usepackage{amsfonts, graphicx}

\newcommand{\be}{\begin{equation}}
\newcommand{\ee}{\end{equation}}
\newcommand{\beq}{\begin{eqnarray}}
\newcommand{\eeq}{\end{eqnarray}}

\newtheorem{thm}{Theorem}[section]

\newtheorem{lma}{Lemma}[section]
\newtheorem{prop}{Proposition}[section]
\newtheorem{cor}{Corollary}[section]

\theoremstyle{remark}
\newtheorem{rem}{Remark}[section]
\numberwithin{equation}{section}

\def\ep{\epsilon}

\def\p{\partial}
\def\S{\Sigma}

\def\R{\mathbb{R}}

\def\tr{{\rm tr}}
\def\p{\partial}
\def\lf{\left}
\def\ri{\right}

\def\R{\Bbb R}
\def\wt{\widetilde}
\def\la{\langle}
\def\ra{\rangle}
\def\bg{\bar{g}}

\def\bg{\bar{g}}

\def\l{\lambda}

\def\Ric{\text{\rm Ric}}
\def\div{\text{\rm div}}

\def\Pi{\overline{\displaystyle{\mathbb{II}}}}

\def\a{\alpha}

\def\bD{\overline{\nabla}}
\def\bnabla{\overline{\nabla}}
\def\bnu{\overline{\nu}}
\def\th{ \tr_{\bg} h}
\def\vbg{d \mathrm{vol}_{\bg} }
\def\vsg{d \sigma_{\bg}}
\def\bH{\bar{H}}

\begin{document}

\title[]{Scalar curvature rigidity with a volume constraint}

\author{Pengzi Miao$^1$}
\address[Pengzi Miao]{School of Mathematical Sciences, Monash University, Victoria 3800, Australia;
Department of Mathematics, University of Miami, Coral Gables, FL 33124, USA.}
\email{Pengzi.Miao@sci.monash.edu.au; pengzim@math.miami.edu}
\author{Luen-Fai Tam$^2$}
\address[Luen-Fai Tam]{The Institute of Mathematical Sciences and Department of
 Mathematics, The Chinese University of Hong Kong,
Shatin, Hong Kong, China.}
\email{lftam@math.cuhk.edu.hk}
\thanks{$^1$ Research partially supported by Australian Research Council Discovery Grant  \#DP0987650
and by a 2011 Provost Research Award of the University of Miami}
\thanks{$^2$Research partially supported by Hong Kong RGC General Research Fund  \#CUHK 403011}
\renewcommand{\subjclassname}{
  \textup{2010} Mathematics Subject Classification}
\date{}

\begin{abstract}
Motivated by Brendle-Marques-Neves' counterexample to the
Min-Oo's conjecture, we prove a volume constrained scalar
curvature rigidity theorem which applies to the hemisphere.
\end{abstract}

\maketitle

\markboth{Pengzi Miao and Luen-Fai Tam}
{scalar curvature rigidity with a volume constraint}

\section{Introduction}
Recently, Brendle, Marques and Neves  \cite{BrendleMarquesNeves}
have solved  the long-standing Min-Oo's conjecture \cite{MinOo}
by constructing a counterexample.

\begin{thm}[Brendle, Marques and Neves \cite{BrendleMarquesNeves}]    \label{thm-BMN}
Suppose $ n \ge 3$.
Let $ \bg $ be the standard metric on the hemisphere $ \mathbb{S}^n_+$.
There exists a smooth metric $g $ on $ \mathbb{S}^n_+$,
which can be made to be arbitrarily close to $ \bg $ in the $ C^\infty$-topology,
satisfying
\begin{itemize}
\item the scalar curvature of $ g $ is at least  that of $\bg$ at each point in $\mathbb{S}^n_+$
\item $ g $ and $ \bg $ agree in a neighborhood of $ \p \mathbb{S}^n_+$,
\end{itemize}
but $ g $ is not isometric to $ \bg$.
\end{thm}

In this paper,  we observe that if the metric $g$ in Theorem \ref{thm-BMN}
is assumed to satisfy  an additional volume constraint,
then it must be isometric to $\bg$. Precisely, we have

 \begin{thm} \label{thm-intro-Vhalfsphere}
Let $ \bg $ be the standard metric on $\mathbb{S}^n_+$.
Let $ g $ be another metric on $ \mathbb{S}^n_+$  with the properties
\begin{itemize}
\item $ R(g) \ge  R(\bg) $ in $ \mathbb{S}^n_+$
\item $ H(g) \ge H(\bg)  $ on $ \p \mathbb{S}^n_+$
\item $ g  $ and $ \bg$ induce the same metric on $ \p \mathbb{S}^n_+$
\end{itemize}
where $ R(g)$, $ R(\bg)$ are the scalar curvature of $ g$, $ \bg$, and
$H (g)$, $H(\bg)$ are the mean curvature of $ \Sigma$ in $(\Omega, g)$,
$(\Omega, \bg)$.
Suppose  in addition
 $$ V(g) \ge V(\bg), $$
where $ V(g)$, $V(\bg)$ are the volume of $ g $, $ \bg$.
If $ || g - \bg ||_{C^2(\bar{\Omega})} $ is sufficiently small, 
then there is a diffeomorphism  $ \varphi: \Omega \rightarrow \Omega$  with  
$\varphi|_{\Sigma}=\text{\rm id}$, the identify map on $ \Sigma$, 
 such that  $\varphi^*(g)=\bg$.
\end{thm}

Theorem \ref{thm-intro-Vhalfsphere} is indeed a special case of a  more general result:

\begin{thm}\label{thm-VRH-rigidity-intro}
Let $(\Omega,\bg)$ be an $ n $-dimensional  compact Riemannian manifold, of  constant sectional
curvature $  1$, with smooth boundary $\S$.
Suppose
 $\Pi + \bH \bar{\gamma}  \ge 0 $
 (i.e $ \Pi + \bH \bar{\gamma}  $ is positive semi-definite), where
$ \bar{\gamma} $ is the induced metric on $ \Sigma$
and $ \Pi $, $ \bH$ are the second
 fundamental form, the mean curvature of $ \Sigma$ in $(\Omega, \bg)$.
Suppose the first nonzero Neumann eigenvalue
$ \mu $  of $(\Omega, \bg)$ satisfies
$\mu>n-\frac2{n+1}$.

Consider a nearby metric $ g $ on  $\Omega$  with the properties
\begin{itemize}
\item  $R(g)\ge n(n-1)$ where $ R(g) $ is the scalar curvature of $ g$
\item $H(g)\ge \bH$ where $H(g)$ is the mean curvature of $\Sigma$ in $(\Omega, g)$
\item $ g $ and $ \bg $ induce the same metric on $ \Sigma$
\item $ V(g) \ge V(\bg)$ where $ V(g)$, $V(\bg)$ are the volumes of $ g$, $ \bg$.
\end{itemize}
If $ || g-\bg ||_{C^2(\bar{\Omega} )} $ is sufficiently small, then there is a diffeomorphism $\varphi$ on $\Omega$ with  $\varphi|_{\Sigma}=\text{\rm id}$,  such that  $\varphi^*(g)=\bg$.
\end{thm}

As a by-product of the method used to derive Theorem \ref{thm-VRH-rigidity-intro}, we  obtain
a  volume estimate for metrics close to the Euclidean metric in terms of the scalar curvature.

\begin{thm} \label{thm-intro-E-volume}
Let $ \Omega \subset \R^n $ be a bounded domain with smooth boundary $ \Sigma$.
Suppose $ \Pi + \bH \bar{\gamma} > 0$ (i.e. $  \Pi + \bH \bar{\gamma} $ is positive definite), where
$ \Pi $, $ \bH $ are the second fundamental form, the mean curvature of $ \Sigma$ in $ \R^n$
and  $ \bar{\gamma} $ is the metric on $ \Sigma$ induced from the Euclidean metric $ \bg$.
Let $ g $ be another metric on $ \bar{\Omega}$ satisfying
\begin{itemize}
\item $ H(g) \ge \bar{H}$, where $ H(g) $ is the mean curvature of $ \Sigma $ in $(\Omega, g)$
\item $ g $ and $ \bg $ induce the same metric on $ \Sigma$.
\end{itemize}
Given any point $ a \in \R^n$,  there exists a constant
$ \Lambda > \frac{ \max_{q \in \bar{\Omega}}  {| q - a |^2} }{4(n-1)}$, depending only on   $\Omega$ and $a $,
such that
if $ || g - \bg ||_{C^3(\bar{\Omega})} $ is sufficiently small, then
\begin{equation} \label{eq-main-formula-V-critical-app-5}
    V(g) - V(\bg)   \ge    \int_\Omega R(g)   \Phi \  \vbg
 \end{equation}
 where
 $ \Phi (x)  =   - \frac{1}{4 (n-1)} |x - a |^2 + \Lambda   > 0  $
 on $ \bar{\Omega}$.
\end{thm}

Theorem \ref{thm-intro-E-volume} may be compared to a previous theorem of Bartnik \cite{Bartnik86},
which estimates the total mass \cite{ADM61} of an asymptotically flat metric that is a perturbation of the
Euclidean metric.

\begin{thm}[Bartnik \cite{Bartnik86}]
Let  $ g $ be an  asymptotically flat metric on $ \R^3$.
If $g$ is sufficiently close to the Euclidean metric $\bg$  (in certain weighted Sobolev space), then
\be
16  \pi \mathfrak{m} (g) \ge \int_{\R^3} R(g) \ \vbg
\ee
where $ \mathfrak{m} (g) $ is the total mass of $ g$. 
\end{thm}

Our proofs of Theorems \ref{thm-intro-Vhalfsphere} - \ref{thm-intro-E-volume}
follow a recent perturbation analysis of Brendle and Marques in  \cite{BrendleMarques},
where they established a  scalar curvature rigidity theorem
for ``small" geodesic balls in $\mathbb{S}^n$.

\begin{thm}[Brendle and Marques \cite{BrendleMarques}]    \label{thm-BM}
Let $\Omega  \subset \mathbb{S}^n $ be a geodesic ball of radius $\delta$.
Suppose
\be \label{eq-condition-geospheres}
\cos\delta\ge \frac{2}{\sqrt{n+3}}.
\ee
Let  $ \bg $ be the standard metric on $ \mathbb{S}^n$. Let  $g$ be another  metric on $\Omega$  with the properties
\begin{itemize}
\item $ R(g) \ge  n (n-1) $ at each point in $ \Omega$
\item $ H(g) \ge \bH   $ at each point on $ \partial \Omega$
\item $ g $ and $ \bg$ induce the same metric on $ \p \Omega$
\end{itemize}
where $ R(g)$ is  the scalar curvature of $ g$, and
$H (g)$, $\bH$ are the mean curvature of $ \p \Omega$ in $(\Omega, g)$,
$(\Omega, \bg)$.
If $ g - \bg $ is sufficiently small in the $ C^2$-norm, then $\varphi^*(g)=\bg$
 for some diffeomorphism  $ \varphi: \Omega \rightarrow \Omega$
 such that  $\varphi|_{\p\Omega}=\text{\rm id}$.
\end{thm}

In Theorem \ref{thm-BM},  the condition \eqref{eq-condition-geospheres} is  equivalently to
 \be \label{eq-condtion-geospheres}
 \bH \ge 4 \tan \delta
 \ee
 because  the mean curvature $\bH$ of $ \p B(\delta) $  is 
 $  (n-1) \frac{ \cos \delta }{ \sin \delta} $.
 As another application  of the formulas in Section  \ref{section-formula},
we obtain a  generalization of Theorem \ref{thm-BM} to
 convex domains in $ \mathbb{S}^n$.

\begin{thm} \label{thm-BM-generalization}
Let  $ \Omega \subset \mathbb{S}^n $  be a smooth domain
contained in a geodesic ball $ B$ of radius less than  $ \frac{\pi}{2} $.
Let $ \bg $ be the  standard metric on $ \mathbb{S}^n$.
Let $ \Pi$, $ \bH$ be the second fundamental form, the mean curvature of $  \p \Omega$
in $(\Omega, \bg)$.
Suppose $ \Omega $ is convex, i.e. $ \Pi \ge 0$.
At $ \p \Omega $, suppose
\be \label{eq-condition-convexS-intro}
\bH \ge 4 \tan{r}
\ee
 where $ r $ is the $\bg$-distance to the center of $ B$.
Then the conclusion of  Theorem \ref{thm-BM} holds on $ \Omega$.
\end{thm}

Theorem \ref{thm-BM-generalization} is an immediate consequence of Theorem \ref{thm-BM-convexS}
in Section \ref{rem-zero}.
In a simpler setting, where  the background metric $ \bg$  is a flat metric,  we have

\begin{thm}\label{thm-intro-RH-zero}
Let $ \Omega $ be a compact manifold with smooth boundary $ \Sigma$.
Suppose there is a flat metric $ \bg $ on $ \Omega$ such that
$ \Pi  + \bH \bar{\gamma} \ge 0$ (i.e. $ \Pi + \bH \bar{\gamma}$ is positive semi-definite),
 where
$ \Pi $, $ \bH$ are the second fundamental form, the mean curvature of $ \Sigma$,
and  $ \bar{\gamma} $ is the induced metric on $ \Sigma$.
Given another metric $ g $ on $ \Omega$ such that
\begin{itemize}
\item $ R(g) \ge  0 $ on $ \Omega$
\item $ H(g) \ge \bH  $ at $ \Sigma $
\item $ g   $ and $ \bg$ induce the same metric on $ \Sigma $,
\end{itemize}
if $ || g - \bg  ||_{C^2(\bar{\Omega})} $ is sufficiently small, then
 $ \varphi^*(g ) = \bg$ for some diffeomorphism $\varphi : \Omega \rightarrow \Omega$
 with  $ \varphi |_{\Sigma} = \mathrm{id} $.
\end{thm}

Similar calculation at the infinitesimal level  provides  examples
of compact  $3$-manifolds of nonnegative scalar curvature
whose  boundary  surface  does not have positive Gaussian curvature but
still has positive Brown-York mass  \cite{BYmass1, BYmass2}.
We include this in the end of the paper to  compare with known results in \cite{ShiTam02}.

\begin{thm} \label{thm-intro-BYmass}
Let $ \Sigma \subset \R^n$ be a connected, closed hypersurface satisfying
$ \Pi + \bH \bar{\gamma} \ge 0$,
 where
$ \Pi $, $ \bH$ are the second fundamental form, the mean curvature of $ \Sigma$,
and  $ \bar{\gamma} $ is the induced metric on $ \Sigma$.
Let $ \Omega$ be the domain enclosed by $ \Sigma$ in $ \R^n$.
Let $ h $ be any nontrivial $(0,2)$ symmetric tensor on $ \Omega $ satisfying
\be \label{eq-conditionh}
\div_{\bg} h = 0, \
\ \tr_{\bg} h = 0, \
h |_{T \Sigma} = 0.
\ee
Let $ \{ g(t) \}_{ | t | < \ep }$ be a $1$-parameter family of metrics on $ \Omega$ satisfying
\be \label{eq-conditiong}
g(0) = \bg, \ \ g^\prime(0) = h, \  \ R (g(t) ) \ge 0 , \
g(t) |_{T \Sigma} = \bg |_{T \Sigma} .
\ee
Then
\be
\int_\Sigma \bH  \vsg > \int_\Sigma  H( g(t) )  \vsg
\ee
for small $ t \neq 0$, where $ H(g(t))$ is the mean curvature of $ \Sigma $ in $(\Omega, g(t))$.
\end{thm}

This paper is organized as follows. In Section \ref{section-formula}, we derive a basic formula
concerning a perturbed metric (Theorem \ref{thm-main-formula-1}),  which corresponds to 
\cite[Theorem 10]{BrendleMarques} of Brendle and Marques.  In Section \ref{rigidity-V}, 
we prove Theorem \ref{thm-VRH-rigidity-intro},  which implies 
Theorem \ref{thm-intro-Vhalfsphere}.  In Section \ref{volume-Rn}, we 
give a proof  of Theorem \ref{thm-intro-E-volume}.  In Section \ref{rem-zero}, 
we consider other applications of the formulas in Section \ref{section-formula} and 
 prove Theorem \ref{thm-BM-generalization} - \ref{thm-intro-BYmass}.

\vspace{.2cm}

{\em Acknowledgment}.  The authors want to thank Simon Brendle and Fernando Marques
for pointing out a false conjecture in a previous draft.  The authors also want to thank the 
referees for useful comments that motivate Theorem 1.7 and Theorem 5.1.

\section{Basic formulas for a perturbed metric} \label{section-formula}

Let $ \Omega $ be an $n$-dimensional, smooth, compact manifold
with boundary $ \Sigma$.  Let $ \bg $ be a fixed smooth Riemannian metric on $ \Omega$.
Given a tensor $ \eta$,  let
  ``$| \eta |$" denote the length of  $\eta$ measured with respect to $ \bg$. Denote the covariant derivative with respect to $\bg$ by $\bD$. Indices of tensors are  raised by $\bg$. Let  $ \bar{R}_{ikjl}$ denote the curvature tensor of $ \bg$ such that if $ \bg $ has constant  sectional curvature $ \kappa$, then
 $  \bar{R}_{ikjl} = \kappa ( g_{ij} g_{kl} - g_{il} g_{kj} ) $.
Consider a nearby  Riemannian metric
$ g = \bg + h $
where $ h $ is a symmetric $(0,2)$ tensor with $  | h| $ very small,
say $ | h | \le \frac12$.

The following pointwise estimates of the scalar curvature of $ g $ and the mean curvature of $ \Sigma $
 were derived by Brendle and Marques in  \cite{BrendleMarques}.

 \begin{prop}[Brendle and Marques \cite{BrendleMarques}]    \label{prop-BM-R}
 The scalar curvatures $R(g) $, $ R(\bg)$ of the metrics $ g$, $ \bg$ satisfy
 \begin{equation*}
 \begin{split}
& \lf |  R(g) -  R (\bg) + \la \Ric(\bg), h \ra - \la \Ric(\bg), h^2 \ra \ri. \\
&  + \frac14 | \bD h |^2  -  \frac12 \bg^{ij} \bg^{kl} \bg^{pq} \bD_i h_{kp} \bD_l h_{jq}
+  \frac14 | \bnabla (\tr_{\bg} h) |^2  \\
 & \lf.    + \bD_i [ g^{ik} g^{jl} ( \bD_k h_{jl} - \bD_l h_{jk} ) ]
\ri | \\
\le  &  \ C \lf( | h | | \bnabla h|^2 + | h |^3 \ri)
 \end{split}
 \end{equation*}
 where  $ \Ric(\bg)$ is the Ricci curvature of $ \bg$,
 $ h^2 $ is the $\bg$-square of $ h$, i.e. $(h^2)_{ik} = \bg^{jl} h_{ij} h_{kl} $,
 $ \la \cdot, \cdot \ra$ is taken with respect to $ \bg$,
  and $ C $ is a positive constant depending only on $ n$.
  \end{prop}

\begin{rem}\label{rem-BM-R}
If the background metric $ \bg$ is Ricci flat, i.e $ \bar{R}_{ik}=0$,
then there will be no $|h|^3$ term in the above estimate. That is because
$$
R(g) =g^{ik} \bar{R}_{ik}  -g^{ik}g^{lj}\lf(\bD_{i,k}h_{jl}-\bD_{i,l}h_{jk}\ri)
+g^{ik}g^{jl}g_{pq}
\lf(\Gamma_{il}^q\Gamma_{jk}^p-\Gamma_{jl}^q\Gamma_{ik}^p\ri),
$$
where each term on the right, except $ g^{ik} \bar{R}_{ik}$, involves derivatives of $h$.
\end{rem}

 \begin{prop}[Brendle and Marques \cite{BrendleMarques}]       \label{prop-BM-H}
 Assume that $ g $ and $ \bg$ induce the same metric on $ \Sigma$, i.e.
 $ h|_{T\Sigma} = 0$ where $ T\Sigma $ is the tangent bundle of $ \Sigma$.
 Then the mean curvatures $H(g)$, $H(\bg)$ of $ \Sigma $
 in $(\Omega, g)$, $(\Omega, \bg)$, each with respect to the outward normals,
 satisfy
 \begin{equation*}
 \begin{split}
& \lf|  2 \lf[ H (g) -  H (\bg) \ri]  -
 \lf(  h(\bnu, \bnu) - \frac14 h(\bnu, \bnu)^2+ \sum_{\alpha=1}^{n-1} h( e_\alpha, \bnu)^2 \ri)
 H(\bg) \ri. \\
 & \lf. + \lf( 1 - \frac12 h(\bnu, \bnu) \ri) \sum_{\alpha=1}^{n-1}  \lf[ 2 \bD_{e_\alpha} h (e_\alpha, \bnu) - \bD_{\bnu} h(e_\alpha, e_\alpha)  \ri]  \ri| \\
 \le & \ C \lf( | h |^2 | \bnabla h | + | h|^3 \ri)
 \end{split}
 \end{equation*}
 where $\{ e_\alpha  \ | \ 1 \le \alpha \le n-1 \}$ is a local orthonormal frame on $ \Sigma$,
 $ \bnu $ is the $\bg$-unit outward normal vector to $ \Sigma$,
 and $ C $ is a positive constant  depending only on $ n$.
\end{prop}

To derive the main formula \eqref{eq-main-formula-1} in this section,  we let
\be\label{linear-R}
DR_{\bg} ( h ) = - \Delta_{\bg} ( \tr_{\bg} h ) + \div_{\bg} \div_{\bg} h - \la \Ric(\bg), h \ra
\ee
be the linearization of the scalar curvature at $ \bg$ along  $ h$. Here
 ``$ \Delta_{\bg}$, $ \div_{\bg} $" denote the Laplacian, the divergence
with respect to $ \bg$.

\begin{lma}\label{lma-R-R} With the same notations  in Proposition \ref{prop-BM-R}, assume in addition $ \div_{\bg} h =0$, then
\begin{equation*}
 \begin{split}
 R(g) - R (\bg) & = DR_{\bg}(h)-\frac12DR_{\bg}(h^2)+\la h,\bD^2\th\ra - \frac14 \lf(| \bD h |^2 +  | \bD (\tr_{\bg} h) |^2\ri)\\&+\frac12h^{ij}h^{kl}\overline R_{ikjl} +E(h) +\bD_i(E_1^i(h))
 \end{split}
 \end{equation*}
 where $ E(h) $ is a function and $ E_1(h)$ is a vector field on $ \Omega$
 satisfying
 $$ | E (h) | \le C ( | h| | \bD h |^2 + | h|^3 ), \ \
  | E_1 (h) | \le C  |h|^2 |\bD h|   $$
for a  positive constant $ C$ depending only on $ n$.
\end{lma}
\begin{proof} First note that
\be\label{linear term}
-\bD_i\lf[\bg^{ik}\bg^{jl}\lf(\bD_k h_{jl}-\bD_lh_{jk}\ri)\ri] -\la \Ric(\bg),h\ra =DR_{\bg}(h).
\ee
Suppose $ g^{ik} = \bg^{ik} + \tau^{ik}$. Then
 $
  \tau^{ik}   = - h^{ik}   + E_2^{ik} (h)
$
where $ h^{ik} = \bg^{ij} h_{jl}  \bg^{lk}$ and  $ | E_2 (h) | \le C | h|^2$.
Hence,
$$
 g^{ik} g^{jl}  - \bg^{ik} \bg^{jl}
 =   - \bg^{ik} h^{jl}   - \bg^{jl} h^{ik}  + E_3^{ikjl}(h)
 $$
 where $ | E_3(h) | \le C | h|^2$.  Therefore,
 \be\label{h^2 term}
 \begin{split}
 -\bD_i &[ ( g^{ik} g^{jl}  - \bg^{ik} \bg^{jl} ) ( \bD_k h_{jl} - \bD_l h_{jk} )]\\
 =&
\bD_i [ (  \bg^{ik} h^{jl}  + \bg^{jl} h^{ik}  -  E_3^{ikjl}(h)) ( \bD_k h_{jl} - \bD_l h_{jk} )]\\
 =& \frac12\Delta_{\bg}|h|^2+\la h,\nabla^2\tr_{\bg}(h)\ra_{\bg}- \div_{\bg}\div_{\bg}(h^2)- \bD_i\lf( E_3^{ikjl}( \bD_k h_{jl} - \bD_l h_{jk} )\ri).
 \end{split}
 \ee
Applying the Ricci identity, one has
 \be
 \begin{split}\label{R-R1}
 \frac12\bg^{ij}\bg^{kl}\bg^{pq}\bD_ih_{kp}\bD_lh_{jq}
 =&\frac12\div_{\bg}\div_{\bg}(h^2)-\frac12\la \Ric(\bg),h^2\ra+\frac12 h^{ij}h^{kl}\overline R_{ikjl}.
 \end{split}
 \ee
The lemma follows from Proposition \ref{prop-BM-R}, \eqref{linear term}, \eqref{h^2 term} and \eqref{R-R1}.
\end{proof}

 Next, let $ DH_{\bg} (h)$ denote the linearization of the
 mean curvature at $ \bg $ along $ h$.
 Proposition \ref{prop-BM-H} implies
 \be \label{eq-linearH-BM}
 DH_{\bg} (h ) = \frac12 \lf[ h(\bnu, \bnu) H(\bg) - \sum_{\alpha=1}^{n-1}
 \lf( 2 \bD_{e_\alpha}  h( e_\alpha, \bnu) - \bD_{\bnu} h (e_\alpha, e_\alpha) \ri) \ri] .
 \ee
For later use, we  note the following equivalent expression of $ DH_{\bg} (h)$
 (see \cite[(34)]{MiaoTam2009} for instance)
\be\label{eq-DH-MT}
 DH_{\bg} (h ) = \frac12 \lf\{ [ d (\tr_{\bg} h) - \div_{\bg} h](\bnu) - \div_\Sigma X \ri\},
 \ee
where $X$ is the vector field on $\S$ dual to the $1$-form $h(\bnu,\cdot) |_{T \Sigma} $.

Let $ DR_{\bg}^* ( \cdot )$ denote the formal $L^2$ $\bg$-adjoint of $ DR_{\bg}(\cdot)$,
i.e.
\be \label{eq-drstar}
DR_{\bg}^* ( \l) = - ( \Delta_{\bg} \l ) \bg + \nabla^2_{\bg} \l - \l \Ric(\bg)
\ee
where  $ \l$ is a function and
 $ \nabla_{\bg}^2 \l $ denotes the Hessian of $ \l $ with
respect to $ \bg$. The content of the following lemma had been used in \cite{MiaoTam2009}.

\begin{lma} \label{lma-DRH}
Let $p$ be any smooth $(0,2)$ symmetric tensor on $\Omega$, then
\be \label{eq-intDR}
\begin{split}
\int_\Omega DR_{\bg} (p) \l \ \vbg = &  \int_\Omega \la DR_{\bg}^* ( \l ), p \ra \ \vbg
- \int_\Sigma 2 DH_{\bg} (p)  \l \ \vsg\\
&+\int_\S \l_{\bnu}\lf(\tr_{\bg}(p)-p(\bnu,\bnu)\ri)\ \vsg
\end{split}
\ee
where $ \l_{\bnu} = \p_{\bnu} \l  $ denotes the directional derivative of $ \l $ along $ \bnu$.
\end{lma}

\begin{proof} Let $Y$ be the vector field on $\S$ dual to the $1$-form $p(\bnu,\cdot)|_{T \Sigma}$.
Integrating by parts, one has
\begin{equation} \label{eq-pf-of-DRRstar}
\begin{split}
& \ \int_\Omega  DR_{\bg} (p) \l \ \vbg - \  \int_\Omega \la DR_{\bg}^* ( \l ), p \ra \ \vbg \\
= & \ \int_\Sigma - \l \p_{\bnu} (\tr_{\bg} p) + (\tr_{\bg} p) \p_{\bnu} \l
+ \l \div_{\bg} p (\bnu)  - p(\bnu, \bnabla \l) \ \vsg \\
= & \ \ \int_\Sigma  \l [ - \p_{\bnu} (\tr_{\bg} p) +  \div_{\bg} p (\bnu) ] -
\la Y, \bnabla^\Sigma \l \ra \ \vsg
+ \int_\S \l_{\bnu}\lf(\tr_{\bg}(p)-p(\bnu,\bnu)\ri)  \ \vsg \\
= & \ \ \int_\Sigma  \l [ - \p_{\bnu} (\tr_{\bg} p) +  \div_{\bg} p (\bnu) + \div_\Sigma Y ]  \ \vsg
+ \int_\S \l_{\bnu}\lf(\tr_{\bg}(p)-p(\bnu,\bnu)\ri) \ \vsg
\end{split}
\end{equation}
where $ \bnabla^\Sigma (\cdot)$ denotes the gradient on $ \Sigma$ with respect to the induced metric. From this and \eqref{eq-DH-MT} the Lemma follows.
\end{proof}

Using Lemma \ref{lma-DRH}, we can  estimate $ \int_\Omega [ R(g) - R (\bg) ] \l \ \vbg $.

\begin{prop} \label{prop-estimate-RbR}
Suppose $ g $ and $ \bg $ induce the same metric on $ \Sigma$
and  $ h $ satisfies
$ \div_{\bg} h = 0 .$
Given any $C^2$ function $ \l $ on $ \Omega$, one has
\begin{equation*}
 \begin{split}
& \   \int_\Omega \lf[R(g) - R(\bg) \ri] \l \  \vbg  \\
  = & \ \int_\Omega \la h, DR^*_{\bg} ( \l) \ra \ \vbg
    -  \frac12 \int_\Omega \la  h^2,  DR^*_{\bg} (\l)  \ra \ \vbg  \\
& \
+  \int_\Omega \lf[ ( \tr_{\bg} h ) \la h, \nabla^2_{\bg} \l \ra   + \frac12   h^{ij} h^{kl} \bar{R}_{ikjl} \l  - \frac14 ( | \bD h |^2 + | \bnabla ( \tr_{\bg} h ) |^2 ) \l  \ri] \ \vbg \\
& \ + \int_\Sigma \lf[ - (h_{nn})^2 - \frac12 | X |^2 \ri] \l_{;n} \ \vsg
 - \int_\Sigma h_{nn} \la X, \bnabla^\Sigma \l \ra  \ \vsg \\
& + \int_\Sigma \lf[
- \frac12  (h_{nn})^2  H(\bg)
  - \frac12 \Pi (X, X) - \frac32  | X|^2 H(\bg)        \ri] \l  \ \vsg
   -   \int_\Sigma  (2 - 2 \tr_{\bg} h )  DH_{\bg}(h)  \l \ \vsg \\
& \  + \int_\Omega E(h) \l  \ \vbg - \int_\Omega E^i_1 (h)  \bD_i \l \ \vbg
+ \int_\Sigma F_1(h) \l \ \vsg
 \end{split}
 \end{equation*}
 where $ \Pi $ is  the second fundamental form of $ \Sigma$ in $(\Omega, \bg)$
with respect to $ \bnu$,  $X$
is the vector field on $\S$ that is dual to the $1$-form $ h (e_n, \cdot)|_{T\Sigma}$,
$E(h)$ and $E_1^i(h)$ are as in Lemma \ref{lma-R-R},  and
$ F_1(h) $ is a  function on $ \Sigma $ satisfying
$$
| F_1(h) | \le C  | h|^2 | \bD h|
$$
for a positive constant $ C $ depending only on $n$.
\end{prop}

\begin{proof}
By \eqref{eq-intDR} with $ p = h $, using  the fact that $h|_{T(\S)}=0$, we have
\be \label{eq-intDR1}
\int_\Omega DR_{\bg} (h) \l \ \vbg =   \int_\Omega \la DR_{\bg}^* ( \l ), h \ra \ \vbg
- \int_\Sigma 2 DH_{\bg} (h)  \l \ \vsg.
\ee
By the second line in  \eqref{eq-pf-of-DRRstar}   with
$ p = h^2$,  and integrating by parts, we also have
\be\label{eq-intDR2}
\begin{split}
& \int_\Omega -\frac\l2 DR_{\bg}(h^2)+\l\la h,\bD^2\th\ra\ \vbg\\
=&\int_\Omega -\frac12 \la DR_{\bg}^*(\l),h^2\ra+\th\la h,\bD^2\l\ra\ \vbg
+\mathcal{B}
\end{split}
\ee
where
\be \label{eq-B}
 \begin{split}
 \mathcal{B}=&
 \int_\Sigma   \frac12 \lf[  \l \p_{\bnu} ( |h|^2 ) -  | h|^2  \p_{\bnu} \l
-  \l (\div_{\bg} h^2) (\bnu)  +  (h^2) (\bnu, \bnabla \l) \ri] \ \vsg  \\
 &+\int_\S \lf[\l h(\bnu,\bD \th)-\th h(\bnu,\bD \l)\ri] \vsg.
 \end{split}
 \ee

To compute $\mathcal{B}$,  let $\{ e_\a \  | \ 1\le \a\le n-1\}$ be an orthonormal frame on $\S$ and let $e_n=\bnu$. Denote $ \bD$ also by `` ; ", thus
$ h_{ij;k} = \bD_k h_{ij} $.
The assumptions $ h |_{T\Sigma} = 0 $
and $ \div_{\bg} h = 0 $  imply the following facts on $ \Sigma$:
\be \label{eq-bdryfact-01}
| h |^2 =  (h_{nn})^2 + 2 | X |^2, \
(h^2)_{nn} = (h_{nn})^2 + | X |^2, \
(h^2)_{n \alpha} = h_{nn} h_{n \alpha} ,
\ee
\be \label{eq-bdryfact-02}
(h^2)(\bnu, \bnabla \l ) =  [ (h_{nn})^2 + | X |^2 ] \l_{;n} + h_{nn} \la X, \bnabla^\Sigma \l \ra,
\ee
\be \label{eq-bdryfact-1}
 h_{\beta \gamma; \alpha} =  h_{\beta n} \Pi_{\gamma \alpha}
+ h_{n \gamma} \Pi_{\beta \alpha},
\ee
\be \label{eq-bdryfact-2}
h_{nn;\alpha} = ( \tr_{\bg} h)_{; \alpha} - \sum_{\beta = 1}^{n-1} h_{\beta \beta; \alpha}
=  ( \tr_{\bg} h)_{; \alpha} - 2 \Pi (X, e_\alpha),
\ee
\be \label{eq-bdryfact-3}
 0 = (\div h)_\alpha = h_{ \alpha n;n} + \sum_{\beta=1}^{n-1} h_{\alpha \beta; \beta}
=
 h_{\alpha n ;n} + h_{n \alpha} H(\bg)
+\Pi (X, e_\alpha) ,
\ee
\be \label{eq-bdryfact-4}
\begin{split}
 0 =  \ ( \div_{\bg} h)_{n}  = h_{nn;n} + \sum_{\alpha=1}^{n-1} h_{n\alpha; \alpha}
=  \ h_{nn;n} + \div_\Sigma X + h_{nn} H (\bg),
\end{split}
\ee
\be  \label{eq-bdryfact-44}
  2 DH_{\bg} (h) = ( \tr_{\bg} h)_{;n} - \div_\Sigma X ,
  \ee
  where \eqref{eq-bdryfact-44} follows from \eqref{eq-DH-MT}.
By \eqref{eq-bdryfact-2}-\eqref{eq-bdryfact-4}, we have
 \be \label{eq-divh2}
 \begin{split}
 &  \ \p_{\bnu} (|h|^2 ) - ( \div_{\bg} h^2 ) (\bnu)  \\
= & \ 3h_{n\a}h_{n\a;n} +h_{nn}h_{nn;n}-h_{n\a}h_{nn;\a}\\
  = &-\Pi(X,X)-3H(\bg)|X|^2-  H(\bg) (h_{nn})^2 -h_{nn} \div_\S X-\la X,\bD^\S \th\ra.
 \end{split}
 \ee
 By \eqref{eq-B}, \eqref{eq-bdryfact-01}, \eqref{eq-bdryfact-02}, \eqref{eq-divh2}
 and integration by parts, we have
\be\label{eq-intDR3}
\begin{split}
 \mathcal{B}
  =&\int_\Sigma \lf[ - (h_{nn})^2 - \frac12 | X |^2 \ri] \l_{;n} -   \int_\Sigma h_{nn} \la X, \bnabla^\Sigma \l \ra \\
& + \int_\Sigma \lf[
  - \frac12 \Pi (X, X) - \frac32  H(\bg)   |X|^2  - \frac12   H(\bg)  (h_{nn})^2   + 2h_{nn}   DH_{\bg} (h)   \ri] \l \vsg.
 \end{split}
\ee
Note that
\be\label{eq-intDR4}
\int_\Omega ( \bD_i E_1^i(h) ) \l  \ \vbg=-\int_\Omega  E_1^i(h)\bD_i\l \ \vbg+\int_\S \l F_1 (h)
\ \vsg
\ee
where $|F_1(h) = \la E_1(h), \bnu \ra  |\le C|h|^2 |\bD h|$.
Proposition \ref{prop-estimate-RbR}    now follows from Lemma \ref{lma-R-R}, \eqref{eq-intDR1}, \eqref{eq-intDR2}, \eqref{eq-intDR3}, and \eqref{eq-intDR4}.
  \end{proof}

The formula \eqref{eq-main-formula-1} next
 is a general form  of \cite[Theorem 10]{BrendleMarques},  which Brendle and Marques derived for geodesic balls in $\mathbb{S}^n$.

\begin{thm} \label{thm-main-formula-1}
Suppose $ g $ and $ \bg $ induce the same metric on $ \Sigma$
and  $ h $ satisfies
$ \div_{\bg} h = 0 .$
Given any $C^2$ function $ \l $ on $ \Omega$, one has
\begin{equation} \label{eq-main-formula-1}
 \begin{split}
& \   \int_\Omega \lf[R(g) - R(\bg) \ri] \l \  \vbg +  \int_\Sigma  ( 2 -  \tr_{\bg} h)
 \lf[H (g) - H (\bg) \ri]  \l \ \vsg \\
  = &  \ \int_\Omega \la h, DR^*_{\bg} ( \l) \ra \ \vbg
    -  \frac12 \int_\Omega \la  h^2,  DR^*_{\bg} (\l)  \ra \ \vbg  \\
& \
+  \int_\Omega \lf[  ( \tr_{\bg} h )  \la h, \nabla^2_{\bg} \l \ra  + \frac12   h^{ij} h^{kl} \bar{R}_{ikjl} \l  - \frac14 ( | \bD h |^2 + | \bnabla ( \tr_{\bg} h ) |^2 ) \l  \ri] \ \vbg \\& \ + \int_\Sigma \lf[ - \frac14 (h_{nn})^2 H(\bg)  - \frac12 ( \Pi(X, X) + H(\bg) |X|^2 ) \ri] \l \ \vsg  \\
& \ + \int_\Sigma \l_{;n} \lf[  - (h_{nn})^2 - \frac12 | X|^2 \ri] \ \vsg
+  \int_\Sigma (-1) h_{nn} \la X, \bnabla^\Sigma \l \ra   \ \vsg \\
& \  + \int_\Omega E(h) \l  \ \vbg + \int_\Omega Z^i (h)  \bD_i \l \ \vbg
+ \int_\Sigma F(h) \l \ \vsg
 \end{split}
 \end{equation}
where $ E(h)$ is a function and $ Z(h) $ is a vector field on $ \Omega$
satisfying
$$
| E(h) | \le C ( | h| | \bD h|^2 + | h |^3 ) ,  \
| Z (h)   | \le C |h|^2 | \bD h|,
$$
and $ F(h) $ is some function on $ \Sigma $ satisfying
$$
| F(h) | \le C ( | h|^2 | \bD h| + | h |^3 ) .$$
\end{thm}

\begin{proof}
Proposition \ref{prop-BM-H} implies
\be \label{eq-H-barH-andDH}
 2 [ H(g) -  H(\bg) ] =  2 DH_{\bg} (h) + J(h) + F_2 (h)
\ee
 where
 $$
 J (h) =   \lf[   \frac14 ( h_{nn})^2 +  | X |^2  \ri] H(\bg)
 - h_{nn}   DH_{\bg} (h)
 $$
and $ F_2 (h) $ is some function on $ \Sigma $  satisfying
$ | F_2 (h) | \le C ( | h|^2 | \bnabla h | + | h |^3 ) $.
Therefore
\be \label{eq-bdry-final}
\begin{split}
& \ ( 2 - h_{nn} ) [ H(g) - H (\bg) ] \\
=  & \ ( 2 - 2 h_{nn} ) DH_{\bg} (h)  +  \lf[   \frac14 ( h_{nn})^2 +  | X |^2  \ri] H(\bg) \\
& \ +  F_2 (h) - \frac12 h_{nn} [ J(h) + F_2 (h) ] .
\end{split}
\ee
\eqref {eq-main-formula-1} now follows readily from Proposition \ref{prop-estimate-RbR}
and \eqref{eq-bdry-final}.
\end{proof}

 The term $ DR^*_{\bg} ( \l ) $ in \eqref{eq-main-formula-1} may
  suggest that one consider a background metric $\bg$ which
  admits a nontrivial function $ \l $ such that $ DR^*_{\bg} (\l) = 0 $
 (such metrics are known as {\em static metrics} \cite{Corvino2000}.)
For instance, if $\bg$ is the standard metric on $ \mathbb{S}^n$
and $ \l  = \cos r $, where $ r $ is the $\bg$-distance to a point,
then  \eqref{eq-main-formula-1} reduces to the formula in
 \cite[Theorem 10]{BrendleMarques}.

Besides static metrics, one can also consider those  metrics $ \bg$
with the property that there exists a function $ \l $ such that
\be \label{eq-critical-metric}
DR^*_{\bg} (\l) = \bg .
\ee
These metrics were studied by the authors in \cite{MiaoTam2009} and \cite{MiaoTam2011}.
 In this case, the terms
 $$
 \int_\Omega \la h , DR^*_{\bg} (\l ) \ra \ \vbg  - \frac12 \int_\Omega \la h^2 , DR_{\bg}^* ( \l ) \ra \ \vbg
 $$
 in \eqref{eq-main-formula-1}  become
 $$
 \int_\Omega \tr_{\bg} h \ \vbg - \frac12 \int_\Omega | h |^2 \ \vbg .
 $$
 To compensate these terms,
 one can include the difference between the volumes of $ g $ and $ \bg$ into \eqref{eq-main-formula-1}.

\begin{cor} \label{cor-main-formula-V-critical}
Suppose $ \bg $ is a metric on $ \Omega$ with the property that
 there exists a function $ \l $ satisfying
 $ DR^*_{\bg} ( \l )  = \bg $.
 Let $ g = \bg + h $ be a nearby metric such that
 $ g $ and $ \bg $ induce the same metric on $ \Sigma$
and  $ h $ satisfies
$ \div_{\bg} h = 0 .$
 Let $ V(g)$, $ V(\bg)$ denote the volume of $(\Omega, g)$, $ (\Omega, \bg)$.
Then
\begin{equation} \label{eq-main-formula-V-critical}
 \begin{split}
& \ -2 ( V(g) - V(\bg) ) +   \int_\Omega \lf[R(g) - R(\bg) \ri] \l \  \vbg
+  \int_\Sigma ( 2 -  \tr_{\bg} h ) \lf[ H (g) - H (\bg) \ri] \l \ \vsg \\
  = & \ \int_\Omega  \lf[ - \frac14  - \frac{1}{n-1} \ri] ( \tr_{\bg} h)^2   \ \vbg
+  \int_\Omega \lf[  - \frac14 ( | \bD h |^2 + | \nabla_{\bg} ( \tr_{\bg} h ) |^2 ) \l  \ri] \ \vbg \\
 & \
+ \int_\Omega \lf[  \frac{1}{1-n} R(\bg) ( \tr_{\bg} h )^2
+ \la h, \Ric(\bg)   \ra  ( \tr_{\bg} h )
  + \frac12   h_{ij} h_{kl} R_{ikjl}  \ri] \l
\ \vbg \\
& \ + \int_\Sigma \lf[ - \frac14 (h_{nn})^2 H(\bg)  - \frac12 ( \Pi (X, X) + H(\bg) |X|^2 ) \ri] \l \ \vsg  \\
& \ + \int_\Sigma \l_{;n} \lf[  - (h_{nn})^2 - \frac12 | X|^2 \ri] \ \vsg
+  \int_\Sigma (-1) h_{nn} \la X, \bnabla^\Sigma \l \ra   \ \vsg \\
& \  + \int_\Omega G(h) \ \vbg +  \int_\Omega E(h) \l  \ \vbg + \int_\Omega Z^i (h) \bD_i \l \ \vbg  + \int_\Sigma F(h) \l \ \vsg
 \end{split}
 \end{equation}
where $ G(h)$ and $ E(h)$ are  functions on $ \Omega$ satisfying
$$ | G (h) | \le C | h |^3 , \ \  | E(h) | \le C ( | h| | \bD h|^2 + | h |^3 ) ,$$
$ Z (h)$ is a vector field on $ \Omega$
satisfying
$$ | Z  (h)   | \le C |h|^2 | \bD h|,  $$
and $ F(h) $ is a  function on $ \Sigma $ satisfying
$$
| F(h) | \le C ( | h|^2 | \bD h| + | h |^3 ) .$$
\end{cor}

\begin{proof}
The difference between the volumes of $ \bg $ and $ g = \bg + h $  is
\be \label{eq-V}
V(g) - V(\bg) = \int_\Omega   \frac12 ( \tr_{\bg} h )
+ \lf[ \frac18 ( \tr_{\bg} h)^2 - \frac14 | h |^2 \ri]    + G(h) \ \vbg ,
\ee
where $ G(h)$ is a function satisfying
$ | G(h) | \le C | h |^3 $
for a  constant $C$ depending only on $ n$.
Suppose $ DR^*_{\bg} ( \l) = \bg $, i.e.
$$
- ( \Delta_{\bg} \l ) \bg + \nabla^2_{\bg} \l - \l \Ric(\bg) = \bg .
$$
Taking trace, one has
$
\Delta_{\bg} \l = \frac{1}{1-n} [ R(\bg) \l + n ].
$
Thus,
\be \label{eq-hessian-V-critical}
\nabla^2_{\bg} \l = \frac{1}{1-n} [ R(\bg) \l  + 1 ] \bg + \l \Ric(\bg) .
\ee
\eqref{eq-main-formula-V-critical} follows  from \eqref{eq-main-formula-1}, \eqref{eq-V}
and  \eqref{eq-hessian-V-critical}.
\end{proof}

\section{volume constrained rigidity}  \label{rigidity-V}

We prove Theorem \ref{thm-VRH-rigidity-intro} in this section.
First, we recall its statement:

\begin{thm}\label{thm-VRH-rigidity}
Let $(\Omega,\bg)$ be an $ n $-dimensional  compact Riemannian manifold, of  constant sectional
curvature $  1$, with smooth boundary $\S$.
Suppose
 $\Pi + \bH \bar{\gamma}  \ge 0 $
 (i.e $ \Pi + \bH \bar{\gamma}  $ is positive semi-definite), where
$ \bar{\gamma} $ is the induced metric on $ \Sigma$
and $ \Pi $, $ \bH$ are the second
 fundamental form, the mean curvature of $ \Sigma$ in $(\Omega, \bg)$.
Suppose the first nonzero Neumann eigenvalue
$ \mu $  of $(\Omega, g)$ satisfies
$\mu>n-\frac2{n+1}$.

Consider a nearby metric $ g $ on  $\Omega$  with the properties
\begin{itemize}
\item  $R(g)\ge n(n-1)$ where $ R(g) $ is the scalar curvature of $ g$
\item $H(g)\ge \bH$ where $H(g)$ is the mean curvature of $\Sigma$ in $(\Omega, g)$
\item $ g $ and $ \bg $ induce the same metric on $ \Sigma$
\item $ V(g) \ge V(\bg)$ where $ V(g)$, $V(\bg)$ are the volumes of $ g$, $ \bg$.
\end{itemize}
If $ || g-\bg ||_{C^2(\bar{\Omega} )} $ is sufficiently small, then there is a diffeomorphism $\varphi$ on $\Omega$ with  $\varphi|_{\Sigma}=\text{\rm id}$, which is the identity map on $ \Sigma$, such that  $\varphi^*(g)=\bg$.
\end{thm}

\begin{proof}
Fix a real number $ p > n$.  By \cite[Proposition 11]{BrendleMarques},
if $ || g - \bg ||_{W^{2,p}({\Omega})} $ is sufficiently small,
 there exists a $ W^{3,p}$ diffeomorphism $\varphi$ on $\Omega$ with   $\varphi |_{\Sigma}=\text{\rm id}$
such that   $ h = \varphi^*(g) - g $ is  divergence free with respect to $ \bg $, and $ || h ||_{W^{2,p}(\Omega) }\le N || g - \bg ||_{W^{2,p} (\Omega)} $
for some positive constant $ N$ depending only on $ (\Omega, \bg)$.
Replacing $ g $ by $ \varphi^*(g)$, we may assume
$ g = \bg + h $ with $ \div_{\bg} h = 0 $. We want to prove that if $ || h ||_{C^1(\bar{\Omega})}$ is sufficiently small and $g$ satisfies the conditions in the theorem, then $h$ must be zero.

Since $\bg$ has constant sectional curvature 1, we choose $\l=-\frac1{n-1}$
such that $DR_{\bg}^*(\l)=\bg$.  Corollary \ref{cor-main-formula-V-critical} then shows
  \begin{equation} \label{eq-VRH-rigidity-1}
 \begin{split}
 & \ -2 ( V(g) - V(\bg) ) - \frac1{n-1}   \int_\Omega \lf[R(g) - R(\bg) \ri] \  \vbg\\
& \ -\frac1{n-1}  \int_\Sigma ( 2 -  \th ) \lf[ H (g) - H (\bg) \ri]  \ \vsg \\
  \ge & \ \frac1{4(n-1)}\int_\Omega  \lf[ - (n+1) (\th)^2  +2|h|^2+ | \bD h |^2 + | \bD ( \th ) |^2 \ri]   \ \vbg \\
& \ + \frac{1}{4(n-1)}\int_\Sigma \lf[   (h_{nn})^2 H(\bg)  +2 ( \Pi (X, X) + H(\bg) |X|^2 ) \ri]   \ \vsg  \\
& \  - C  || h ||_{C^1 ( \bar{\Omega} )  } \lf[ \int_\Omega ( | h |^2 + | \bD h |^2 ) \ \vbg
  + \int_\Sigma | h |^2 \ \vsg \ri]
 \end{split}
 \end{equation}
 for a constant $C$ depending only on $(\Omega, \bg)$.

Using the variational property of $ \mu$, we have
\be \label{eq-mu-new}
\int_\Omega | \bD ( \th ) |^2 \ \vbg \ge \mu
\lf[ \lf( \int_\Omega ( \th )^2 \ \vbg \ri) - \frac{1}{V(\bg)} \lf( \int_\Omega \th \ \vbg \ri)^2 \ri] .
\ee
By  \eqref{eq-V},  $ \int_\Omega \th \ \vbg$ is related to $ (V(g) - V(\bg))$  by
\be \label{eq-V-2}
 \int_\Omega \th \ \vbg =
 2 ( V(g) - V(\bg) )
-  \int_\Omega  \lf\{ \lf[ \frac14 ( \tr_{\bg} h)^2 - \frac12 | h |^2 \ri] + 2 G(h) \ri\}   \ \vbg ,
\ee
where $ G(h) \le C | h |^3$.

Given any constant $ 0 < \ep < 1$,  using  \eqref{eq-mu-new} and  the fact
 $|h|^2\ge\frac1n (\th)^2 $ and $|\bD h|^2\ge \frac1n|\bD (\th) |^2$, we have
\be \label{eq-getep}
\begin{split}
& \int_\Omega  \lf[ - (n+1) (\th)^2  +2|h|^2+ | \bD h |^2 + | \nabla_{\bg} (\th) |^2 \ri]   \ \vbg\\
\ge &  \int_\Omega \lf[   \ep |h|^2 + \ep | \bD h |^2  +   \lf[ - (n+1) + \frac{2 - \ep}{n} \ri]  (\th)^2 + \lf[ \frac{( 1 - \ep)}{n}  + 1 \ri] | \bD (\th) |^2  \ri]   \ \vbg\\
\ge & \int_\Omega \lf[   \ep |h|^2 + \ep | \bD h |^2  +   \lf[ - (n+1) + \frac{2 - \ep}{n}
+ \frac{( 1 - \ep)}{n} \mu + \mu \ri]   (\th)^2  \ri]  \ \vbg\\
& -  \mu \lf[ \frac{( 1 - \ep)}{n}  + 1 \ri] \frac{1}{V(\bg)} \lf( \int_\Omega \th  \ \vbg \ri)^2 .
\end{split}
\ee
Since $ \mu > n - \frac{2}{n+1} $, we can chose $ \ep $ (depending only on $\mu$ and $n$)  such that
\be \label{eq-choseep}
 \lf[ - (n+1) + \frac{2 - \ep}{n}
+ \frac{( 1 - \ep)}{n} \mu + \mu \ri]  \ge 0 .
\ee
Then it follows from \eqref{eq-V-2},  \eqref{eq-getep} and \eqref{eq-choseep}
that
\be \label{eq-getep2}
\begin{split}
& \int_\Omega  \lf(- (n+1) ( \th) ^2  +2|h|^2+ | \bD h |^2 + | \bD ( \th  ) ) |^2 \ri)   \ \vbg\\
\ge & \ \ep \int_\Omega \lf(   |h|^2 + | \bD h |^2 \ri)  \ \vbg
 -C_1 ( V(g) - V(\bg) )^2 -C_1\int_\Omega |h|^4 \ \vsg\\
\end{split}
\ee
where $ C_1$ is a positive constant  depending only on $ (\Omega, \bg)$.

At the boundary $ \Sigma$, the  assumption  $ \Pi + H(\bg) \bar{\gamma} \ge 0$
implies $ H(\bg) \ge 0 $, therefore
\be \label{eq-bdry-new}
\int_\Sigma \lf[   (h_{nn})^2 H(\bg)  +2 ( \Pi (X, X) + H(\bg) |X|^2 ) \ri]   \ \vsg \ge 0
\ee
for any $ h$.
By \eqref{eq-VRH-rigidity-1}, \eqref{eq-getep2} and \eqref{eq-bdry-new},  we have
  \begin{equation} \label{eq-VRH-rigidity-2}
 \begin{split}
 & \ - 8 (n-1) ( V(g) - V(\bg) ) - 4   \int_\Omega \lf[R(g) - R(\bg) \ri] \  \vbg\\
& \  - 4  \int_\Sigma ( 2 -  \th ) \lf[ H (g) - H (\bg) \ri]  \ \vsg \\
  \ge & \ {\ep} \int_\Omega \lf( |h|^2+   |\bD h|^2\ri)\ \vbg
 \\
& \ -C ( V(g) - V(\bg) )^2 -C \int_\Omega |h|^4  \ \vbg \\
  & \ - C  || h ||_{C^1 ( \bar{\Omega} )  } \lf[ \int_\Omega ( | h |^2 + | \bD h |^2 ) \ \vbg
  + \int_\Sigma | h |^2 \ \vsg \ri]
   \end{split}
 \end{equation}
 for some positive constant $ C$ depending only on $ (\Omega, \bg)$.

 Finally,  we note that
  \be \label{eq-V-3}
  \lf( V(g) - V(\bg) \ri)^2
\le    C \lf( \int_\Omega | h | \ \vbg \ri)  ( V(g) - V(\bg) )
 \ee
  by \eqref{eq-V-2} and the assumption $ V (g) \ge V(\bg)$.
 Also, by the trace theorem,
 \be\label{eq-trace bound}
 ||h||_{L^2(\Sigma)}\le C||h||_{W^{1,2}(\Omega)}
  \ee
  for a constant $ C$ only depending on $ \Omega$.
Therefore,  by  \eqref{eq-VRH-rigidity-2}, \eqref{eq-V-3}, \eqref{eq-trace bound} and the
assumptions $ V(g) \ge V(\bg)$, $ R(g) \ge R(\bg) $ and $ H(g) \ge H(\bg)$ , we conclude that
if $ || h ||_{C^1 (\bar{\Omega}) } $ is sufficiently small,  then
 \be
 0 \ge \frac{\ep}{2}    \int_\Omega ( | h |^2 + | \bD h |^2 ) \ \vbg
 \ee
 which implies $ h $ must be identically zero.
 This completes the proof.
  \end{proof}

\begin{rem}
In Theorem \ref{thm-VRH-rigidity},  if $ \Sigma$ is indeed empty, i.e $(\Omega, \bg)$ is a closed space form,
its first nonzero Neumann eigenvalue  satisfies   $\mu \ge n$ as
$(\Omega, \bg) $ is covered by $ \mathbb{S}^n$.
In this  case, Theorem \ref{thm-VRH-rigidity} says  that {$V (g) \ge V(\bg)$
implies $ g $ is isometric to $ \bg$ for a nearby metrics $ g$ with $R(g)\ge R(\bg) $}.
This could be compared to a more profound theorem known in $3$-dimension:
``{\em If $(M,g)$ is closed $3$-manifold with $R(g)\ge 6$, $\Ric(g)\ge g$ and  $V(g)\ge V(\mathbb{S}^3)$, then $(M,g)$ is isometric to $\mathbb{S}^3$}." (See \cite[Corollary 5.4]{Brendle} and earlier reference of   \cite{Bray-thesis, GurskyViaclovsky2004})
\end{rem}

When $ \Sigma \neq \emptyset $,  the boundary assumption $ \Pi + \bH \bar{\gamma} \ge 0 $ in Theorem
\ref{thm-VRH-rigidity} can be relaxed
in certain circumstances. A detailed examination of the above proof shows,  if
\be
 \Pi(v,v) + \bH  \bar{\gamma} \ge  - \beta \bar{\gamma}
\ee
for some positive constant $ \beta$, where $ \beta $ is sufficiently small comparing to
the constant $ \ep $ in \eqref{eq-choseep} and the constant $ C$ in
\eqref{eq-trace bound}, then the conclusion of Theorem \ref{thm-VRH-rigidity}  still
holds on such an $ (\Omega, \bg)$.
In particular, this shows

\begin{cor} \label{cor-VRH-rigidity}
Let $(M, \bg)$ be an $n$-dimensional Riemannian manifold of constant sectional curvature $ 1$. Suppose $ \Omega \subset M$ is a bounded domain with smooth boundary $ \Sigma$, satisfying the assumptions in Theorem \ref{thm-VRH-rigidity},
 i.e $ \mu > n - \frac{2}{n+1} $ and
$ \Pi + \bH \bar{\gamma} \ge 0 $ on $ \Sigma$.
Let $ \tilde{\Omega} \subset M $ be another bounded domain with smooth boundary
$ \tilde{\Sigma}$. If $ \tilde{\Sigma} $ is sufficiently close to $ \Sigma$ in the
$ C^2 $ norm, then the conclusion of Theorem \ref{thm-VRH-rigidity}  holds
on $ \tilde{\Omega}$.
\end{cor}

It is known that the fist nonzero Neumann eigenvalue of $ \mathbb{S}^n_+$ is $n$ (see \cite[Theorem 3]{Chavel}). Therefore, Theorem \ref{thm-intro-Vhalfsphere} follows from Theorem \ref{thm-VRH-rigidity}.
Moreover, by Corollary \ref{cor-VRH-rigidity},
Theorem \ref{thm-VRH-rigidity}  holds on a geodesic ball in $ \mathbb{S}^n $ whose
 radius  is slightly larger than $ \frac{\pi}{2} $.

 By the next lemma, we know Theorem \ref{thm-VRH-rigidity} also  holds on any geodesic ball in $ \mathbb{S}^n$ that is  strictly contained in $ \mathbb{S}^n_+$.

\begin{lma}\label{lma-mu-estimate}
Let $ B(\delta) \subset \mathbb{S}^n$ be a geodesic ball of radius $ \delta$.
Let $ \mu (\delta) $ be the first nonzero Neumann eigenvalue of $ B(\delta)$.
\begin{enumerate}
\item[(i)] $ \mu( \delta) $ is a strictly decreasing function of $ \delta$ on $(0, \frac{\pi}{2}]$.
\item[(ii)] For any $ 0< \delta< \frac{\pi}{2}$,
$$ \mu (\delta) > n + \frac{ (\sin \delta)^{n-2} \cos \delta}{ \int_0^\delta ( \sin  t)^{n-1} d t }
>  \frac{n}{ (\sin \delta)^{2} }.$$
\end{enumerate}
\end{lma}
\begin{proof}
By  \cite[Theorem 2, p.44]{Chavel}, $ \mu(\delta)$ is characterized by the fact
  that
 \be \label{e-J-1}
\lf\{  ( \sin t)^{n-1} J^\prime \ri\}^\prime + [ \mu (\delta)  - (n-1) (\sin t)^{-2} ] ( \sin t)^{n-1} J = 0
\ee
has a solution $ J = J(t) $  on $ [0 , \delta ] $ satisfying
\be \label{e-J-condition}
J(0) = 0 , \ J^\prime ( \delta ) = 0, \ J^\prime(t) \neq 0, \ \forall \ t \in [0, \delta) .
\ee

Given $ 0 < \delta_1 < \delta_2 \le \frac{\pi}{2} $,   let $ J_i = J_i (t) $ be a solution to
\eqref{e-J-1} with $ \mu(\delta) $ replaced by $ \mu(\delta_i)$,
satisfying \eqref{e-J-condition} on $ [0, \delta_i]$, $ i = 1, 2$.
Replacing $ J_i $ by $ - J_i $ if necessary,  we may
assume that $J'_i>0$ on $[0,\delta_i)$, hence $J_i>0$ on $(0,\delta_i]$.
Define
$$
f_i=\frac{ (\sin t)^{n-1} J'_i}{J_i}, \
\beta_i(t)=\lf[ \mu(\delta_i) - \frac{n-1}{ (\sin t)^2}\ri] (\sin t)^{n-1}.
$$
By \eqref{e-J-1}, $f_i $ satisfies
$$
f'_i = - \beta_i -\frac{1}{ (\sin t) ^{n-1} }f_i^2.
$$
Therefore, on $(0, \delta_1]$,
\be\label{eq-f1-f2}
( f_1 - f_2)^\prime = \frac{1}{ ( \sin t)^{n-1}  }(f_2^2-f^2_1)
+ [ \mu ( \delta_2) - \mu (\delta_1) ] ( \sin t)^{n-1}.
\ee
Note that $ f_1(t)$, $f_2(t)$ can be extended continuously to $ 0 $ such that
$ f_1 (0) = f_2(0) $. Moreover,  $ f_1>0$, $f_2>0$  on  $(0,\delta_1)$,
$ f_2(\delta_1)>0 = f_1(\delta_1) $.
Let $0\le t_0<\delta_1$ be such that $f_1=f_2$ at $t_0$ and $f_2>f_1$ for $t_0< t \le \delta_1$. On $(t_0,\delta_1]$, one would have
$ (f_1-f_2)^\prime >0 $
if $ \mu (\delta_2) \ge \mu(\delta_1) $, which is a contradiction to $ f_2 > f_1$.
Therefore,  $ \mu ( \delta_2) < \mu (\delta_1) .$ This proves (i).

To prove (ii), we further claim that $ t_0 = 0 $, i.e. $f_2>f_1$ on $(0,\delta_1]$.
If not,   there would  be a nonpositive local  minimum  of $ (f_2-f_1)$ at some
$ \tilde{t}_0\in (0, t_0]$.
At $ \tilde{t}_0$,  \eqref{eq-f1-f2} implies
\be
0=( f_1 - f_2)^\prime\le [ \mu ( \delta_2) - \mu (\delta_1) ] ( \sin \tilde{t}_0)^{n-1}<0
\ee
because $ 0 < f_2 ( \tilde{t}_0) \le f_1 ( \tilde{t}_0 ) $ and
$\mu(\delta_2)<\mu(\delta_1)$. Hence $f_2>f_1$ on $(0,\delta_1]$.
Integrating \eqref{eq-f1-f2} on $[0, \delta_1]$,  we  have
\be \label{eq-mu-integrate}
-f_2(\delta_1)= \int_0^{\delta_1}( f_1 - f_2)^\prime dt
>    [ \mu ( \delta_2) - \mu (\delta_1) ] \int_0^{\delta_1}( \sin t)^{n-1}dt.
\ee
Therefore
\be
 \mu (\delta_1)> \mu(\delta_2)+\frac{f_2(\delta_1)}{\int_0^{\delta_1}( \sin t)^{n-1}dt}.
\ee
Now let $\delta_1=\delta \in (0, \frac{\pi}{2})$ and $\delta_2=\pi/2$.
Applying the fact that $ \mu (\frac{\pi}{2} ) = n $, $J_2 = \sin t $, and
$$
f_2= (\sin t)^{n-2} \cos t ,
$$
we have
\be
\begin{split}
\mu (\delta)> & \ n+\frac{ (\sin \delta)^{n-2} \cos \delta}{\int_0^{\delta}( \sin t)^{n-1}dt}\\
>   & \ n+\frac{(\sin \delta)^{n-2}  \cos^2 \delta}{ \int_0^{\delta}\cos t( \sin t)^{n-1}dt}\\
=  & \ \frac{n}{\sin^2\delta} .
\end{split}
\ee
Therefore, (ii) is proved.
\end{proof}

\section{A Volume estimate on domains in $\mathbb{R}^n$}\label{volume-Rn}
On $ \R^n$, the standard   Euclidean  metric
$ \bg $ satisfies $ DR^*_{\bg} ( \l ) = \bg $ with
\be
 \l (x) =   - \frac{1}{2 (n-1)}  | x - a |^2 + L
  \ee
 where $ | \cdot |  $ denotes the Euclidean length,
  $ a \in \R^n $ is any fixed point and $ L$ is an  arbitrary constant. In this section,
we  use this fact and  Corollary \ref{cor-main-formula-V-critical}
   to prove Theorem \ref{thm-intro-E-volume} in the introduction.
 First we need some lemmas.

\begin{lma} \label{lma-extension}
On a compact Riemannian manifold $(\Omega, \bg)$ with
smooth boundary $ \Sigma$, there exists a positive constant $ C$ depending only
on $(\Omega, \bg)$ such that,  for any Lipschitz function $\phi$ on $\Sigma$, there is an extension of $\phi$ to a Lipschitz function $\wt \phi$  on $\Omega$ such  that
\be \label{eq-extension}
\int_\Omega \lf(|\wt \phi|^2+|\bnabla\wt \phi|^2\ri) \vbg \le C\int_\Sigma \lf(\phi^2+|\bnabla^\Sigma\phi|^2\ri) \vsg
\ee
where $ \bnabla $, $ \bnabla^\Sigma$ denote the gradient on $ \Omega$, $ \Sigma$
respectively.
\end{lma}

\begin{proof} Let $ d (\cdot, \Sigma)$ be the distance to $ \Sigma$.
Let $\delta>0$ be a small constant such that the tubular neighborhood
$U_{2\delta}=\{x\in \Omega|\ d(x,\S)<2\delta\} $
can be parametrized by    $F:\S\times [0,2\delta)\to U_{2 \delta}$, with
$F(y, t) = exp_y ( t \nu(y))$ where $ \exp_y (\cdot)$ is the exponential map
at $ y \in \Sigma $ and $ \nu(y) $ is the inward unit normal at $ y$.
In $ U_{2 \delta}$, the metric $ \bg $ takes the form $ dt^2+\sigma^t$, where
$ \{ \sigma^t \}_{ 0\le t < 2\delta} $ is a family of metrics on $\S$.
By choosing $ \delta $ sufficiently small, one can assume
$ \sigma^t $ is equivalent to  $ \sigma^0 $ in the sense that
$ \frac12 \le \sigma^t (v,v) \le 2 $ for any tangent vector $ v $ with $ \sigma^0(v, v) = 1$,
 $ \forall \ 0 \le t < 2\delta$.

Let $ \rho = \rho (t) $ be  a fixed smooth  cut-off function  on $[0, \infty)$ such that
$ 0 \le \rho \le 1 $, $ \rho (t ) = 1 $ for $ 0 \le t \le \delta$ and
$ \rho (t) = 0 $ for $ t \ge \frac32 \delta$.
 On $ U_{2 \delta}$, consider the function $ \wt{\phi}(y, t) = \phi(y) \rho (t) $.
 Since $ \wt{\phi} $ is identically zero outside
 $ U_{\frac32 \delta} = \{x\in \Omega|\ d(x,\S)<\frac32\delta\} $,
 $ \tilde{\phi} $ can be viewed as an extension
 of $ \phi $ on $ \Omega$. For such an $ \wt{\phi}$, one has
\be \label{eq-exten-1}
\begin{split}
\int_\Omega |\wt\phi|^2 \vbg \le  \int_0^{2\delta} \lf( \int_{\Sigma}|\phi|^2 d\sigma^t \ri) dt
\le  \ C\delta\int_\S|\phi|^2 \vsg
\end{split}
\ee
and
\be \label{eq-exten-2}
\begin{split}
\int_\Omega |\bnabla\wt\phi|^2 \vbg \le &
\ 2\int_{U_{2\delta}} \lf( |\bnabla \rho|^2\phi^2+|\bnabla\phi|^2 \rho^2 \ri) \vbg \\
\le & \ C\delta\int_\S|\phi|^2 \vsg +
2\int_0^{2\delta} \lf( \int_{\Sigma}|\bnabla^{\Sigma}_t\phi|^2 d\sigma^t \ri) dt  \\
\le & \  C\lf[\int_\S|\phi|^2 \vsg +\int_\S|\bnabla^\Sigma \phi|^2 \vsg \ri]
\end{split}
\ee
where $ \bnabla^\Sigma_t $ denotes the gradient on $(\Sigma,   \sigma^t)$
and $ C$ is a positive constant depending only on $(\Omega, \bg)$.
\eqref{eq-extension} now follows from \eqref{eq-exten-1} and \eqref{eq-exten-2}.
\end{proof}

\begin{lma} \label{lma-estofh3}
On  a compact Riemannian manifold $(\Omega, \bg)$ with
smooth boundary $ \Sigma$, there exists a positive constant
$ C$ depending only on $ (\Omega, \bg)$ such that,
for any smooth $(0,2)$ symmetric tensor $ h$ on $ \Omega$, one has
\be \label{eq-estofh3}
\int_\Omega | h|^3 \vbg \le
C \lf( \int_\Sigma |h|^3 \vsg + || h ||_{C^2(\Omega)} \int_\Sigma |h|^2 \vsg
+ \int_\Omega | h | | \bnabla h |^2 \vbg \ri)
\ee
\end{lma}

\begin{proof} On $ \Omega$, let $ \phi =|h|^\frac32$. By lemma \ref{lma-extension},
there exists  a Lipschitz function $ \wt \phi $ on $ \Omega$ such that
 $ \wt \phi |_{\Sigma} = \phi |_{\Sigma} $ and
$$
 \int_\Omega \lf(|\wt \phi|^2+|\bnabla\wt \phi|^2\ri) \vbg \le C\int_\Sigma \lf(\phi^2+|\bnabla^\Sigma\phi|^2\ri) \vsg .
 $$
Let $ \lambda_1 > 0 $ be the first Dirichlet eigenvalue of $ (\Omega, \bg)$, then
\be \label{eq-usingexten}
\begin{split}
 \  \int_\Omega\phi^2 \ \vbg
\le & \ 2  \int_\Omega \lf[ \wt\phi^2 + (\phi-\wt\phi)^2 \ri] \vbg \\
\le & \ 2  \int_\Omega  \wt\phi^2 \  \vbg + 2 \lambda_1^{-1} \int_\Omega | \bnabla (\phi-\wt\phi) |^2 \vbg \\
\le& \ C\lf[\int_\S \lf(\phi^2+|\bnabla^\Sigma \phi|^2\ri) \vsg +\int_\Omega|\bnabla \phi|^2 \vbg \ri]
\end{split}
\ee
where
\be \label{eq-goodint}
\begin{split}
\int_\Omega | \bnabla \phi |^2 \vbg = \int_\Omega | \bnabla |h|^\frac32 |^2 \vbg
\le \frac94 \int_\Omega | h| | \bnabla h |^2 \vbg .
\end{split}
\ee

To handle the boundary term $ \int_\Sigma  | \bnabla^\Sigma \phi |^2 \vsg $,
given any constant $ \ep > 0$, one considers
\be \label{eq-intbypart}
 \int_\Sigma   | \bnabla^\Sigma  ( |h|^2 + \ep )^\frac34  |^2  \vsg
= - \int_\Sigma  ( |h|^2 + \ep )^\frac34 \Delta_\Sigma   ( |h|^2 + \ep )^\frac34\vsg
\ee
where $ \Delta_\Sigma $ denotes the Laplacian on $ \Sigma$.
Let $ \{ e_\alpha \ | \ \alpha = 1, \ldots, n-1 \}$ be a local orthonormal
frame on $ \Sigma$ and $ e_n $ be the outward unit normal to $ \Sigma$.
Let $ \bar{H} $ be the mean curvature of $ \Sigma$ with respect to $ e_n$.
Denote covariant differentiation $ \Omega$ by `` $;$  ".
Let $ i$, $j$ run through $\{1, \ldots, n\}$.
One has
\be
\begin{split}
\Delta_\Sigma |h|^2 = & \sum_{\alpha} ( |h|^2)_{;\alpha \alpha} - \bH ( | h|^2)_{;n}  \\
= &  \sum_{\alpha, i, j, } 2 (  h_{ij} h_{ij;\alpha \alpha}  + h_{ij;\alpha }^2 )  - \bH
\sum_{i,j} 2 h_{ij} h_{ij;n} \\
\ge & - C || h ||_{C^2(\bar{\Omega})} | h|.
\end{split}
\ee
 Therefore,
\be \label{eq-bdDelta}
\begin{split}
\Delta_\Sigma (  |h|^2 + \ep )^\frac34 = &
\frac34 ( | h|^2 + \ep )^{- \frac14} \Delta_\Sigma | h |^2
- \frac{3}{16} ( | h |^2 + \ep )^{- \frac54} | \bnabla^\Sigma | h|^2 |^2 \\
\ge & - C || h ||_{C^2(\bar{\Omega})}   ( | h|^2 + \ep )^{- \frac14} |h|
- \frac{3}{16} ( | h |^2 + \ep )^{- \frac54} | \bnabla^\Sigma | h|^2 |^2.
\end{split}
\ee
It follows from \eqref{eq-intbypart} and \eqref{eq-bdDelta} that
 \be
 \begin{split}
 \int_\Sigma   | \bnabla^\Sigma  ( |h|^2 + \ep )^\frac34  |^2  \vsg
\le & \ C || h ||_{C^2(\bar{\Omega})} \int_\Sigma ( | h |^2 + \ep)^{  \frac12} | h | \vsg \\
&  +  \frac{1}{3} \int_\Sigma | \bnabla^\Sigma ( | h|^2 + \ep )^\frac34 |^2
 \vsg .
\end{split}
 \ee
 Letting $ \ep \rightarrow 0$, one has
  \be \label{eq-bdrygoodint}
 \int_\Sigma   | \bnabla^\Sigma  |h|^\frac32  |^2  \vsg
\le  C || h ||_{C^2(\bar{\Omega})} \int_\Sigma | h |^2  \vsg .
 \ee
 \eqref{eq-estofh3} now follows from \eqref{eq-usingexten},
 \eqref{eq-goodint} and \eqref{eq-bdrygoodint}.
\end{proof}

We recall the statement of Theorem \ref{thm-intro-E-volume} and give its proof.

\begin{thm} \label{thm-intro-E-volume-later}
Let $ \Omega \subset \R^n $ be a bounded domain with smooth boundary $ \Sigma$.
Suppose $ \Pi + \bH \bar{\gamma} > 0$ (i.e. $  \Pi + \bH \bar{\gamma} $ is positive definite), where
$ \Pi $, $ \bH $ are the second fundamental form, the mean curvature of $ \Sigma$ in $ \R^n$
and  $ \bar{\gamma} $ is the metric on $ \Sigma$ induced from the Euclidean metric $ \bg$.
Let $ g $ be another metric on $ \bar{\Omega}$ satisfying
\begin{itemize}
\item $ g $ and $ \bg $ induce the same metric on $ \Sigma$.
\item $ H(g) \ge \bar{H}$, where $ H(g) $ is the mean curvature of $ \Sigma $ in $(\Omega, g)$.
\end{itemize}
Given any point $ a \in \R^n$,  there exists a constant
$ \Lambda > \frac{ \max_{q \in \bar{\Omega}}  {| q - a |^2} }{4(n-1)}$,
which depends only on   $\Omega$ and $a $,
such that
if $ || g - \bg ||_{C^3(\bar{\Omega})} $ is sufficiently small, then
\begin{equation} \label{eq-main-formula-V-critical-app-5-later}
    V(g) - V(\bg)   \ge    \int_\Omega R(g)   \Phi \  \vbg
 \end{equation}
 where $ \Phi  =   - \frac{1}{4 (n-1)} |x - a |^2 + \Lambda  > 0
 $ on $ \bar{\Omega}$.
\end{thm}

\begin{proof}
Fix a number $ p > n $.
By the proof of   \cite[ Proposition 11]{BrendleMarques},
one knows if $||g-\bg ||_{W^{3,p}(\Omega)}$ is sufficiently small,
then  there exists a $W^{4,p}$ diffeomorphism $\varphi : \Omega \rightarrow \Omega $
such that  $\varphi|_{\Sigma}=\text{\rm id}$, $ h=\varphi^*(g)-\bg $
is divergence free with respect to $ \bg$, and
$ || h ||_{W^{3, p}(\Omega)} \le N || g - \bg ||_{W^{3, p}(\Omega)} $
for a positive constant $ N $ depending only on $ (\Omega, \bg) $.
In what follows, we will work with  $ \phi^*(g)$.
For convenience, we still denote $ \phi^*(g) $ by $ g$.

Given $ a  \in \R^n$, consider
$
\l (x) =   - \frac{1}{2 (n-1)} |x - a |^2 + L
$
where $ L $ is a constant to be determined.
First, we require
$
L > \frac{1}{2(n-1)} \max_{q \in \bar{\Omega}}  {| q - a |^2}
$
so that $ \l > 0 $ on $ \bar{\Omega}$. Since
$ \l $ satisfies  $ DR_{\bg}^* ( \l) = \bg $,
Corollary \ref{cor-main-formula-V-critical} shows
\begin{equation} \label{eq-main-formula-V-critical-app}
 \begin{split}
& \ -2 ( V(g) - V(\bg) ) +   \int_\Omega R(g) \l \  \vbg
+  \int_\Sigma ( 2 -  \tr_{\bg} h ) \lf[ H (g) - \bH \ri] \l \ \vsg \\
  \le  &   - \int_\Omega    \frac14  | \bD h |^2  \l  \ \vbg
   + \int_\Sigma \lf[ - \frac14 (h_{nn})^2 \bH  - \frac12 ( \Pi (X, X) + \bH  |X|^2 ) \ri] \l \ \vsg  \\
& \ + \int_\Sigma \l_{;n} \lf[  - (h_{nn})^2 - \frac12 | X|^2 \ri] \ \vsg
+  \int_\Sigma (-1) h_{nn} \la X, \bnabla^\Sigma \l \ra   \ \vsg \\
& \  + \int_\Omega G(h) \ \vbg +  \int_\Omega E(h) \l  \ \vbg + \int_\Omega Z^i (h) \bD_i \l \ \vbg  + \int_\Sigma F(h) \l \ \vsg
 \end{split}
 \end{equation}
where
$ | G (h) | \le C | h |^3$, $  | E(h) | \le C ( | h| | \bD h|^2 + | h |^3 ) ,$
$ | Z  (h)   | \le C |h|^2 | \bD h| $, $ | F(h) | \le C ( | h|^2 | \bD h| + | h |^3 ) $
for some constant $ C$ depending only on $\Omega$.

At  $ \Sigma$,  $  \l_{;n}$ and $\bD^\Sigma \l $
are determined solely by $ \Omega$ and $a$ (in particular they are independent on $L$).
Apply the assumption $ \Pi + \bar{H} \bar{\gamma} > 0$ (which  implies $\bH >0$)
and the fact $ | h |^2 = ( h_{nn} )^2 + 2 | X |^2$,
we have
\be \label{eq-bdry-indepenL}
\begin{split}
& \lf[ - \frac14 (h_{nn})^2 \bH   - \frac12 ( \Pi (X, X) + \bH  |X|^2 ) \ri] \l  \\
& +
\l_{;n} \lf[  - (h_{nn})^2 - \frac12 | X|^2 \ri]
+   (-1) h_{nn} \la X, \bnabla^\Sigma \l \ra \\
\le & - L C_1 | h |^2   + C_2 | h |^2
\end{split}
\ee
where $ C_1 $, $ C_2$  are  positive constants depending only on $ \Omega$ and $a$.
We fix $ L  $  such that
\be \label{eq-conditionL}
L C_1   - C_2  > 0
\ee
and let $ m = \frac14 \min_{\bar{\Omega}} \l  $ (note that $\l$ is fixed now).
 \eqref{eq-main-formula-V-critical-app}-\eqref{eq-conditionL} imply
\begin{equation} \label{eq-main-formula-V-critical-app-2}
 \begin{split}
& \ -2 ( V(g) - V(\bg) ) +   \int_\Omega R(g)  \l \  \vbg
+  \int_\Sigma ( 2 -  \tr_{\bg} h ) \lf[ H (g) - \bH  \ri] \l \ \vsg \\
  \le  & \ -  m\int_\Omega   | \bD h |^2   \ \vbg  - \lf( L C_1   - C_2  \ri) \int_\Sigma | h |^2 \vsg \\
& \   + C_3 \lf( \int_\Omega ( | h| | \bD h|^2 + | h |^3 )  \vbg +
   \int_\Sigma  ( | h|^2 | \bD h| + | h |^3 ) \ \vsg \ri)
 \end{split}
 \end{equation}
where $ C_3$ depends only on $ \Omega$, $a$ and $L$.
Apply Lemma \ref{lma-estofh3} to the term $ \int_\Omega | h |^3 \ \vbg$
on the right side of \eqref{eq-main-formula-V-critical-app-2}, we  have
\begin{equation*} \label{eq-main-formula-V-critical-app-4}
 \begin{split}
& \ -2 ( V(g) - V(\bg) ) +   \int_\Omega R(g)  \l \  \vbg
+  \int_\Sigma ( 2 -  \tr_{\bg} h ) \lf[ H (g) - \bH  \ri] \l \ \vsg \\
  \le  & \ -  m\int_\Omega   | \bD h |^2   \ \vbg  -  ( L C_1   - C_2 ) \int_\Sigma | h |^2 \vsg  \\
  & \  + C || h ||_{C^2(\bar{\Omega})} \lf( \int_\Omega   | \bD h|^2    \vbg +
   \int_\Sigma  | h|^2    \ \vsg \ri) .
 \end{split}
 \end{equation*}
where $ C$ is independent on $h$.
From this, we conclude that if $ || h ||_{C^2 (\bar{\Omega} )  } $ is sufficiently
 small, then \eqref{eq-main-formula-V-critical-app-5-later} holds with $ \Phi = \frac12 \l $.
 This completes the proof.
\end{proof}

\begin{rem}
When $ \Omega \subset \R^n $ is a ball of radius $R$, one can take $a$ to be the center of $ \Omega$. In this
case, by computing $ \bH$, $ \Pi$ and $ \l_{;n}$ explicitly in \eqref{eq-conditionL},
  the constant $ L$  can be chosen to be any constant satisfying
$$
L >  \lf[ \frac{1}{2(n-1)} + \frac{4}{(n-1)^2} \ri] R^2 .
$$
\end{rem}

 \begin{rem}
 By the results in \cite{Miao02, ShiTam02}
based on the positive mass theorem \cite{SchoenYau79, Witten81},  a metric $g$
on $ \Omega$ satisfying the boundary conditions in  Theorem \ref{thm-intro-E-volume-later}
must be isometric to the Euclidean metric  if  $R(g) \ge 0 $.  Therefore, a nontrivial
 metric $ g$ in Theorem \ref{thm-intro-E-volume-later} necessarily has negative  scalar curvature
 somewhere. For such a $g$, Theorem \ref{thm-intro-E-volume-later} shows
  if the weighted   integral  $ \int_\Omega R(g) \Phi \ \vbg $
 is nonnegative, then $ V(g) \ge V(\bg)$.
 \end{rem}

\section{Other related results}  \label{rem-zero}

In this  section, we collect some other by-products  of the formulas derived in Section 2.
First, we discuss a scalar curvature rigidity result for  general domains in $\mathbb{S}^n$.

\begin{thm} \label{thm-BM-convexS}
Let  $ \Omega \subset \mathbb{S}^n $  be a smooth domain
contained in a geodesic ball $ B$ of radius  less than $  \frac{\pi}{2} $.
Let $ \bg $ be the  standard metric on $ \mathbb{S}^n$.
Let $ \Pi$, $ \bH$ be the second fundamental form, the mean curvature of $ \Sigma = \p \Omega$
in $(\Omega, \bg)$ with respect to the outward unit normal $ \bnu$.
Suppose   $ \Pi \ge -c \bar{\gamma}$, where $c\ge 0$ is a function on $\Sigma$
and $ \bar{\gamma}$ is the induced metric on $ \Sigma$.
Let $ q $ be the center of $ B$.
Suppose at $ \Sigma\setminus\{q\}$,
\be \label{eq-condition-convexS}
\bH-c \ge \left[ \frac{5 \cos \theta + \sqrt{ \cos^2 \theta + 8 }  }{2} \right] \tan{r} 
\ee
where $ r $ is the $\bg$-distance to $q$ and $ \theta $ is the angle between $ \bnu $ and $ \bnabla r$.
Then the conclusion of Theorem \ref{thm-BM} holds on $ \Omega$.
\end{thm}

\begin{proof} As before,  replacing $ g $ by $ \varphi^* (g) $ for some
diffeomorphism $\varphi $, we may assume
$ \div_{\bg} h = 0 $ where $ h = g - \bg$.
On $\Omega$, let  $ \l = \cos r > 0$,  where $ r$ is the $ \bg$-distance to $ q$.
At   $ \Sigma \setminus \{ q \}$, we have
\be
 \l_{;n}  =   - \sin r \cos \theta, \ \
  | \bD^\Sigma \l | = \sin r  \sin \theta  .
\ee

 Apply Theorem \ref{thm-main-formula-1},  using the fact $ DR^*_{\bg} ( \l ) = 0 $ and the  assumptions
on $ R(g) $ and $ H(g)$, we have
\begin{equation} \label{eq-condition-convexS-2}
 \begin{split}
 & \  \int_\Omega \lf[   \frac14 ( | \bD h |^2 + | \bnabla ( \tr_{\bg} h) |^2 )
 + \frac12\lf(|h|^2+ (\tr_{\bg} h)^2\ri)  \ri]  \cos r
   \ \vbg \\
\le & \  \int_\Sigma \lf[ - \frac14 (h_{nn})^2  \bH  - \frac12 ( \Pi(X, X) + \bH  |X|^2 ) \ri] \cos r \ \vsg  \\
& \ + \int_{\Sigma \setminus \{ q \} } \lf[  (h_{nn})^2 + \frac12 | X|^2  \ri]  ( \sin{r}  \cos \theta )   \ \vsg
+ \int_{\Sigma \setminus \{ q \} }   | h_{nn}|  | X |  (\sin r  \sin \theta )   \ \vsg   \\
& \ +  C || h ||_{ C^1 ( \bar{ \Omega } ) }  \lf\{  \int_\Omega ( |  h |^2 + | \bD h|^2  ) \ \vbg
+ \int_\Sigma  |h|^2    \ \vsg  \ri\}\\
\le & \ - \int_{\Sigma \setminus \{ q \} } \bigg[  \lf( \frac14 ( \bH - c ) \cos r-\sin r\cos \theta\ri) (h_{nn})^2 + \frac12 \lf(  (\bH-c)\cos r-\sin r\cos\theta\ri)  |X|^2\\
  & \ -  | h_{nn}|  | X |  (\sin r  \sin \theta ) \bigg]  \ \vsg   \\
& \ +  C || h ||_{ C^1 ( \bar{ \Omega } ) }  \lf\{  \int_\Omega ( |  h |^2 + | \bD h|^2  ) \ \vbg
+ \int_\Sigma  |h|^2    \ \vsg  \ri\}
 \end{split}
 \end{equation}
 for some positive constant $ C$ independent on $ h$.

 Note that  the assumption \eqref{eq-condition-convexS} implies
\be \label{eq-condition-convexS-consequence-1}
 \frac14 (\bH-c) \cos r - (\sin r \cos \theta) \ge 0
\ee
and
\be \label{eq-condition-convexS-consequence-2}
(\bH-c) \cos r   -    (\sin r \cos \theta ) \ge 0.
\ee
 By \eqref{eq-condition-convexS}, \eqref{eq-condition-convexS-consequence-1}
and \eqref{eq-condition-convexS-consequence-2}, we have
\be \label{eq-conclusion-convexS}
\begin{split}
0 \le & \  \lf( \frac14 (\bH-c) \cos r - \sin r \cos \theta \ri) (h_{nn})^2  -  | h_{nn}|  | X | (\sin r \sin \theta )  \\
&  + \frac12 \lf(   (\bH-c) \cos r   -   \sin r \cos \theta  \ri)  |X|^2
\end{split}
\ee
 for any $ h_{nn} $ and $ X$.
The result now follows from  \eqref{eq-condition-convexS-2} and  \eqref{eq-conclusion-convexS}.
\end{proof}

\begin{rem}
It is clear from the proof of Theorem \ref{thm-BM-convexS} that the center $ q$ of $B$ does not 
need to be inside $ \Omega$.
\end{rem}

Theorem \ref{thm-BM-convexS}  directly implies   Theorem \ref{thm-BM-generalization} in the introduction.

\begin{proof}[Proof of Theorem \ref{thm-BM-generalization}] Choose $c=0$ in Theorem \ref{thm-BM-convexS}. Since
$$ \displaystyle 4 \ge  \frac{5 \cos \theta + \sqrt{ \cos^2 \theta + 8 }  }{2} $$
for any $\theta$,  the result follows from Theorem \ref{thm-BM-convexS}. 
\end{proof}

Next, we consider a corresponding scalar curvature rigidity result when the background metric $\bg$ is a flat metric.

\begin{thm}\label{thm-RH-zero}
Let $ \Omega $ be a compact manifold with smooth boundary $ \Sigma$.
Suppose $ \bg $ is a smooth Riemannian metric on $ \Omega$ such that
$ \bg $ has zero sectional curvature and
$ \Pi  + \bH \bar{\gamma} \ge 0 $ on $ \Sigma$,
 where
$ \Pi $, $ \bH$ are the second fundamental form, the mean curvature of $ \Sigma$,
and  $ \bar{\gamma} $ is the induced metric on $ \Sigma$.
Suppose $ g $ is another metric on $ \Omega$ satisfying
\begin{itemize}
\item  $R(g)\ge 0$ where $ R(g) $ is the scalar curvature of $ g$
\item $ g $ and $ \bg $ induce the same metric on $ \Sigma$
\item $H(g)\ge \bH$ where $H(g)$ is the mean curvature of $\Sigma$ in $(\Omega, g)$ .
\end{itemize}
If $ || g - \bg ||_{C^2 (\bar{\Omega} )} $ is sufficiently small, then
there is a diffeomorphism $ \varphi$ on $ \Omega$ with
$ \varphi |_{\Sigma} = \mathrm{id} $
such that $ \varphi^*(g ) = \bg$.
\end{thm}

\begin{proof} As before,  we may assume $ \div_{\bg} h = 0$ where $ h  = g - \bg$.
Choose $ \l = 1 $ in \eqref{eq-main-formula-1},
one has
\begin{equation}  \label{eq-zeromfd}
 \begin{split}
& \    \int_\Omega \lf[   \frac14 ( | \bD h |^2 + | \bnabla ( \tr_{\bg} h ) |^2 )  \ri]
   \ \vbg \\
& \ + \int_\Sigma \lf[  \frac14 (h_{nn})^2 H(\bg)  + \frac12 ( \Pi(X, X) + H(\bg) |X|^2 ) \ri]  \ \vsg  \\
\le & \  \int_\Omega E(h)  \ \vbg  + \int_\Sigma F(h) \ \vsg
 \end{split}
 \end{equation}
where $ | F(h) | \le C ( | h|^2 | \bD h| + | h |^3 ) $ and
$ | E(h) | \le C   | h| | \bD h|^2   $  by Remark \ref{rem-BM-R}.
The result  follows from \eqref{eq-zeromfd}.
\end{proof}

To finish, we mention that  the positive Gaussian curvature condition of the boundary surface
 in \cite{ShiTam02}  is not a necessary condition for the positivity of its Brown-York mass.

\begin{thm} \label{thm-BYmass}
Let $ \Sigma \subset \R^n$ be a connected, closed hypersurface satisfying
$ \Pi + \bH \bar{\gamma} \ge 0$,
 where
$ \Pi $, $ \bH$ are the second fundamental form, the mean curvature of $ \Sigma$,
and  $ \bar{\gamma} $ is the induced metric on $ \Sigma$.
Let $ \Omega$ be the domain enclosed by $ \Sigma$ in $ \R^n$.
Let $ h $ be any nontrivial $(0,2)$ symmetric tensor on $ \Omega $ satisfying
\be \label{eq-conditionh-later}
\div_{\bg} h = 0, \
\ \tr_{\bg} h = 0, \
h |_{T \Sigma} = 0.
\ee
Let $ \{ g(t) \}_{ | t | < \ep }$ be a $1$-parameter family of metrics on $ \Omega$ satisfying
\be \label{eq-conditiong-later}
g(0) = \bg, \ \ g^\prime(0) = h, \  \ R (g(t) ) \ge 0 , \
g(t) |_{T \Sigma} = \bg |_{T \Sigma} .
\ee
Then
\be
 \int_\Sigma \bH  \vsg >  \int_\Sigma  H( g(t) ) )  \vsg
\ee
for small $ t \neq 0$, where $ H(g(t))$ is the mean curvature of $ \Sigma$ in $(\Omega, g(t) )$.
\end{thm}

\begin{proof} By Lemma \ref{lma-DRH},  one knows
$$
\frac{d }{d t }
  \lf(   \int_\Omega \lf[R(g(t)) - R(\bg) \ri]  \  \vbg -  2 \int_\Sigma [ \bH  - H( g(t) ) ] \ \vsg \ri) \Big|_{t = 0}=0.
$$
Direct calculation using Lemma \ref{lma-DRH}, \eqref{eq-bdryfact-3} and \eqref{eq-conditionh-later}
shows
\be
\begin{split}
\ & \frac{d^2}{d t^2}
  \lf(   \int_\Omega \lf[R(g(t)) - R(\bg) \ri]  \  \vbg -
 2 \int_\Sigma [ \bH  - H( g(t) ) ] \ \vsg   \ri) \Big|_{t = 0} \\
= &     - \frac12 \int_\Omega  | \bD h |^2  \ \vbg  - \int_\Sigma \lf[   ( \Pi(X, X) + H(\bg) |X|^2 ) \ri]  \ \vsg  \\
 \end{split}
 \end{equation}
which is negative by the assumption on $ \Pi + \bH \bar{\gamma}$.
Thus, for small $ t $,
\be
 2 \int_\Sigma [ \bH  - H( g(t) ) ] \ \vsg  > \int_\Omega \lf[R(g(t)) - R(\bg) \ri] \ \vbg \ge 0.
\ee
\end{proof}

Given an $ h $ satisfying \eqref{eq-conditionh},
a family of deformation $\{g(t) \}$ satisfying \eqref{eq-conditiong}
is given by
$
g(t) = u(t)^\frac{4}{n-2} ( \bg + t h)
$
for small $ t $, where $ u(t) > 0 $ is a conformal factor such that $ R(g(t)) = 0 $
(see \cite[Lemma 4]{MiaoTam2009}).

An example of a non-convex surface $ \Sigma \subset \R^3 $,
which is topologically a $2$-sphere and satisfies the
 condition
$ \Pi + \bH \bar{\gamma} \ge 0 $,
is given by a capsule-shaped surface
with its middle slightly  pinched.

\end{document}